\documentclass[english,12pt]{elsarticle}
\usepackage[english]{babel}
\usepackage{geometry}
\usepackage{hyperref}
\usepackage{amsthm}
\usepackage{amsmath}
\usepackage{amssymb}
\usepackage{graphicx}
\usepackage{thmtools, thm-restate}
\usepackage{cleveref}
\makeatletter
\theoremstyle{plain}
\newtheorem{theorem}{Theorem}[section]

\newtheorem{lemma}{Lemma}[section]
\newtheorem{remark}{Remark}[section]
\newtheorem{proposition}{Proposition}[section]

\usepackage{babel}
\usepackage{tikz}
\usepackage{tikz-cd}
\numberwithin{equation}{section}
\usepackage{caption}
\usepackage{comment}

\usepackage{accents}
\newcommand{\dbtilde}[1]{\accentset{\approx}{#1}}

\usepackage{etoolbox}
\apptocmd{\sloppy}{\hbadness 10000\relax}{}{}

\let\today\relax
\makeatletter
\def\ps@pprintTitle{%
	\let\@oddhead\@empty
	\let\@evenhead\@empty
	\def\@oddfoot{\footnotesize\itshape
		\hfill\today}%
	\let\@evenfoot\@oddfoot
}
\makeatother

\usepackage{dynkin-diagrams}

\begin{document}

\begin{frontmatter}

	\title{Hopf superalgebra structure for Drinfeld super Yangian of Lie superalgebra B(m,n)}

	\author[1]{Alexander Mazurenko}
	\ead{mazurencoal@gmail.com}

	\address[1]{MCCME (Moscow Center for Continuous Mathematical Education)}

	\begin{abstract}
		We construct a Hopf superalgebra structure for a Drinfeld super Yangian of Lie superalgebra $B(m,n)$ relative to all possible choices of Borel subalgebras. In order to do this we introduce a simplified definition of Drinfeld super Yangians.
	\end{abstract}


\end{frontmatter}

\section*{Acknowledgements}

This work is funded by Russian Science Foundation, scientific grant 21-11-00283. With gratitude to Stukopin V.A., who introduced the author to the world of superalgebras.

\vspace{10pt}

\section{Introduction}

\vspace{10pt}

The Drinfeld super Yangian, a concept rooted in the rich theory of superalgebras and Lie superalgebras, has gained substantial attention in recent mathematical research. To learn more details about the structure of Drinfeld super Yangians, refer to the following works: \cite{S04}, \cite{M21a}, \cite{A02}, \cite{Z96}, \cite{U19}. In particular, the exploration of its Hopf superalgebra structure has been an area of significant interest. This paper delves into precisely this area of study, focusing on the Drinfeld super Yangian associated with the Lie superalgebra $B(m,n)$.

The concept of the Drinfeld super Yangian for the Lie superalgebra $A(m, n)$, which serves as a precursor to our exploration, has been introduced relative to specific subclasses of Borel subalgebras, as detailed in \cite{MS22}. Building on this groundwork, we extend our focus to the Lie superalgebra $B(m, n)$. The formal definition of the Drinfeld super Yangian for $B(m, n)$ relative to distinguished Borel subalgebra is presented in \cite{M21}, establishing a foundation for our current study.

One of the central objectives of this paper is to provide insights how to generalize our results across a broader spectrum of basic Lie superalgebras, encompassing all possible choices of Borel subalgebras. This extension necessitates the introduction of additional relations in the definitions of both Lie superalgebras and Drinfeld super Yangians, along with supplementary lemmas. However, the fundamental approach to our proofs remains consistent.

Our paper is structured as follows:

Section \ref{sec:prelim} - Preliminaries: This section provides essential background information, including fundamental algebraic concepts (Subsection \ref{subs:LS}) and general definitions and results pertaining to Lie superalgebras $B(m, n)$ (Subsection \ref{subsc:bs}).

Section \ref{sec:DSY} - Drinfeld Super Yangians: Here, we introduce the fundamental definitions and concepts associated with Drinfeld super Yangians. Subsection \ref{sect:DSY} presents the definition of the Drinfeld super Yangian of Lie superalgebra $B(m, n)$ relative to all possible choices of Borel subalgebras. In Subsection \ref{sb:mpsy}, we provide a minimalistic representation of the Drinfeld super Yangian, a crucial prerequisite for constructing the Hopf superalgebra structure discussed in Section \ref{sec:hssdsy}.

Section \ref{sec:hssdsy} - Hopf Superalgebra Structure: This section presents the core result of our paper. In Theorem \ref{th:hssdsy}, we establish the Hopf superalgebra structure for the Drinfeld super Yangian of Lie superalgebra $B(m, n)$, extending our findings from the minimalistic representation presented in Section \ref{sb:mpsy}.

Section \ref{sec:par} - Proof of Theorem \ref{th:mpy}: The final section is dedicated to proving Theorem \ref{th:mpy}, which is foundational for our exploration of the Hopf superalgebra structure. Subsection \ref{sub:grdsy} introduces special elements necessary for constructing the minimalistic presentation, while Subsection \ref{sub:pmpt} contains auxiliary lemmas crucial for establishing the theorem.

This structured approach allows us to build a comprehensive understanding of the Hopf superalgebra structure of the Drinfeld super Yangian for Lie superalgebra $B(m, n)$, and it serves as a stepping stone for broader generalizations and applications in the field.

\vspace{10pt}

\section{Preliminaries}

\label{sec:prelim}

\vspace{10pt}

We introduce general notations, definitions and state important results about Lie superalgebras.

\vspace{10pt}

\subsection{Notations}

\label{subs:LS}

\vspace{10pt}

By $\mathbb{N}$, $\mathbb{N}_{0}$, $\mathbb{Z}$ we denote the sets of natural numbers, natural numbers with zero and integers respectively. Recall that $\mathbb{Z}_{2} = \{ \bar{0}, \bar{1} \}$. Let $\Bbbk$ be an algebraically closed field of characteristic zero. Denote by $\mathcal{M}_{n,m}(A)$ a ring of $n \times m$ $(n,m \in \mathbb{N} )$ matrices over a ring $A$. We also use the Iverson bracket defined by $ [P] = \begin{cases} 1 \text{ if } P \text{ is true;} \\ 0 \text{ otherwise}, \end{cases}$ where $P$ is a statement that can be true or false. $[1,n]$ denotes $\{1, 2,  ... ,n\}$ for any $n \in \mathbb{N}$.

Note that all equations in this subsection can be extended by linearity from homogeneous to all elements in corresponding algebraic structure.

Let $A$ be an associative superalgebra. Denote for all homogeneous $x \in A$ the $\mathbb{Z}_2$-grading of $x$ by $|x| \in \mathbb{Z}_2$. One defines
\begin{enumerate}
	\item the Lie superbracket by
	\begin{equation}
		\label{eq:LS}
		[x,y] := xy - (-1)^{|x||y|} y x
	\end{equation}
	\item the anticommutator by
	\[ \{x,y\} := xy + (-1)^{|x||y|} yx; \]
	\item the graded Lie superbracket by
	\[ [x,y]_{v} := xy - (-1)^{|x||y|} v y x \]
\end{enumerate}
for all homogeneous elements $x, y \in A$ and all $v\in \Bbbk$.

Consider super vector spaces $V$ and $W$. Define the linear function $\tau_{V, W}: V \otimes W \to W \otimes V$ by
\begin{equation}
	\label{eq:taudef}
	\tau_{V, W}(v \otimes w) = (-1)^{|v| |w|} w \otimes v
\end{equation}
for all homogeneous $v \in V$ and $w \in W$.

A Lie superalgebra is a super vector space $\mathfrak{g} = \mathfrak{g}_{\bar{0}} \oplus \mathfrak{g}_{\bar{1}}$ together with a bilinear map (the Lie superbracket) $[\cdot, \cdot]: \mathfrak{g} \times \mathfrak{g} \to \mathfrak{g}$ which satisfies the following axioms:
\begin{equation}
	[ g_{ \bar{ \alpha }}, g_{ \bar{\beta}} ] \subseteq g_{\bar{ \alpha } + \bar{\beta}} \text{ for } \bar{ \alpha }, \bar{\beta} \in \mathbb{Z}_{2},
\end{equation}	
\begin{equation}
	[x, y] = - (-1)^{|x| |y|} [y, x],
\end{equation}
\begin{equation}
	\label{eq:sje}
	(-1)^{|x||z|} [x,[y,z]] + (-1)^{|x||y|} [y,[z,x]] + (-1)^{|z||y|}[z,[x,y]] = 0
\end{equation}
for all homogeneous $x, y, z \in \mathfrak{g}$.

The linear mapping $\textnormal{ad}_{x}: \mathfrak{g} \to \mathfrak{g}$ is called the adjoint action and is defined by $\textnormal{ad}_{x}(y) = [x,y]$ for all $x, y \in \mathfrak{g}$. Denote by $\textnormal{id}_{\mathfrak{g}}$ the identity map on $\mathfrak{g}$.

We use results about formal power series proved in \cite{N69}. Let $A$ be a ring. Consider the ring of formal power series $A[[X]]$. Define for each $f \in 1+XA[[X]]$
\[ \log(f) := \sum_{n=1}^{\infty} \frac{(-1)^{n+1}}{n} (f-1)^{n} \in XA[[X]]. \]

\vspace{10pt}

\subsection{Lie Superalgebras}

\label{subsc:bs}

\vspace{10pt}

Let $\mathfrak{g}$ be the Lie superalgebra $B(m,n) = \mathfrak{osp}(2m+1|2n)$ with $m 
\ge 0$, $n \ge 1$ relative to a Borel subalgebra $\mathfrak{b}$. Recall that $\mathfrak{g}$ has rank $m+n$ and set $I = \{1,2,..,m+n\}$. Suppose that $\mathfrak{g}$ has a simple root system $\Delta^{0} = \{ \alpha_{i} \}_{i \in I}$ and a Cartan subalgebra $\mathcal{H} = \langle h_{i} \rangle_{i \in I}$. Denote the corresponding simple root generators by $ e_{i}$ and $f_{i}$ for all $i \in I$. For $\tau \subseteq I$, we define a $\mathbb{Z}_{2}$-gradation by setting $|e_{i}| = \bar{0}$ ($|f_{i}| = \bar{0}$) if $i \notin \tau$ else $|e_{i}| = \bar{1}$ ($|f_{i}| = \bar{1}$) for all $i \in I$. Moreover, set a $\mathbb{Z}_{2}$-gradation on $I$ by setting $|i| = \bar{0}$ if $i \notin \tau$ else $|i| = \bar{1}$ for all $i \in I$.

Decompose $\mathfrak{g}$ using its canonical root decomposition as follows:
\[ \mathfrak{g} = \mathcal{H} \oplus \bigoplus_{ \alpha \in \Delta } \mathfrak{g}^{ \alpha } \]
where $\Delta$ is a root system of $\mathfrak{g}$ and
\[ \mathfrak{g}^{ \alpha } = \{ x \in \mathfrak{g} \; | \; [h,x] = \alpha(h) x, h \in \mathcal{H} \} \]
is the root space corresponding to a root $\alpha \in \Delta$. It is well known that $ \dim(\mathfrak{g}^{ \alpha }) = 1$ as a super vector space for all $\alpha \in \Delta$. Thus we can fix basis elemenets $ \{  e_{\alpha}, f_{\alpha} \}_{ \alpha \in \Delta^{+} } $ such that $\mathfrak{g}^{ \alpha } = \langle e_{\alpha} \rangle$ and $\mathfrak{g}^{ -\alpha } = \langle f_{\alpha} \rangle$ for all $\alpha \in \Delta^{+}$. Define $\textnormal{ht}( \sum_{i \in I} n_{i} \alpha_{i} ) = \sum_{i \in I} n_{i}$ for all $( n_{i} )_{i \in I} \in \mathbb{Z}^{|I|}$ where $\alpha_{i} \in \Delta^{0}$ for all $i \in I$.

By definition of basic Lie superalgebra (see \cite{K77}) there exists a unique (up to a constant factor) even nondegenerate $\mathfrak{g}$-invariant supersymmetric bilinear form $\langle . , . \rangle : \mathfrak{g} \times \mathfrak{g} \to \Bbbk $, i. e.
\begin{enumerate}
	\item $\langle \mathfrak{g}_{\bar{\alpha}} , \mathfrak{g}_{\bar{\beta}} \rangle = 0$ unless $\bar{\alpha} + \bar{\beta} = \bar{0}$ for $\bar{\alpha}, \bar{\beta} \in \mathbb{Z}_{2}$;
	\item the form induces an isomorphism $\mathfrak{g} \cong (\mathfrak{g})^{*}$;
	\item $ \langle [x,y] , z \rangle = \langle x , [y,z] \rangle$ for all $x,y,z \in \mathfrak{g}$;
	\item $ \langle x , y \rangle = (-1)^{|x||y|} \langle y , x \rangle $ for all homogeneous $x,y \in \mathfrak{g}$.
\end{enumerate}
Without loss of generality we choose a bilinear form on $\mathfrak{g}$ in such a way that
\begin{enumerate}
	\item $\langle e_{i}, f_{j} \rangle = \delta_{i,j}$ for all $i,j \in I$;
	\item $\langle e_{\alpha}, f_{\beta} \rangle = [\alpha + \beta = 0]$ for all $\alpha, \beta \in \Delta$;
	\item the restriction of $\langle . , . \rangle$ to  $\langle \mathfrak{g}^{\alpha} , \mathfrak{g}^{-\alpha} \rangle$ is nondegenerate for all $\alpha \in \Delta$. 
\end{enumerate}
Using the last property note that for any $\alpha \in \mathcal{H}^{*}$ there exists a unique element $h_{\alpha} \in \mathcal{H}$ such that $\langle h_{\alpha} , h \rangle = \alpha(h)$ for all $h \in \mathcal{H}$. Thus we can define a nondegenerate symmetric bilinear form $( . , .)$ on $\mathcal{H}^{*}$ by $( \alpha, \beta) = \langle h_{\alpha} , h_{\beta} \rangle $ for all $\alpha, \beta \in \mathcal{H}^{*}$. Now we are able to construct the symmetric Cartan matrix $C = (c_{ij})_{i,j \in I}$ of $\mathfrak{g}$ by supposing $C = ( (\alpha_{i}, \alpha_{j}) )_{i,j \in I}$. Therefore the following arithmetic conditions for elements of the Cartan matrix $C$ hold:
\begin{enumerate}
	\item $\langle h_{i}, h_{j} \rangle = (\alpha_{i}, \alpha_{j}) = c_{ij}$ for all $i,j \in I$;
	\item $c_{ij} \in \mathbb{Z}$, for all $i, j \in I$;
	\item $c_{ij} \le 0$, if $i \notin \tau$, for all $i, j \in I$.
\end{enumerate}
Let $\{ \epsilon_{i}, \delta_{j} \; | \; i \in [1,m], j \in [1,n] \}$ be the dual basis of $\mathcal{H}^{*}$. Then for $\mathfrak{g}$ the bilinear form is defined by
\begin{align}
	( \epsilon_{i}, \epsilon_{i^{'}} ) = \delta_{i,i^{'}}, \; (\delta_{j}, \delta_{j^{'}}) = -\delta_{j,j^{'}}, (\epsilon_{i}, \delta_{j}) = 0
\end{align}
for all $i, i^{'} \in [1,m]$ and $j, j^{'} \in [1,n]$. We also require the following well-known
\begin{lemma}
	\leavevmode
	\begin{enumerate}
		\item If $x \in \mathfrak{g}^{\alpha}$, $y \in \mathfrak{g}^{-\alpha}$, then $[x,y] = \langle x,y \rangle h_{\alpha}$ for all $\alpha \in \Delta$.
		\item $ [ \mathfrak{g}^{\alpha}, \mathfrak{g}^{-\alpha} ] = \langle h_{\alpha} \rangle $ for all $\alpha \in \Delta$.
		\item $ [ h, \mathfrak{g}^{\alpha} ] = \pm \langle h, h_{\alpha} \rangle \mathfrak{g}^{\alpha} $ for all $\alpha \in \Delta^{\pm}$.
	\end{enumerate}
\end{lemma}

Using all previous definitions we formulate the well-known result (for more details see \cite{Y91}, \cite{Y94}, \cite{Z11}):
\begin{proposition}
	\label{pr:supliealgdef}
	The Lie superalgebra $\mathfrak{g}$ is generated by elements $h_{i}$, $e_{i}$ and $f_{i}$ for $i \in I$ subject to the relations
	\begin{description}
	\item[LS1] for $i, j \in I$
	\begin{equation}
		[h_{i}, h_{j}] = 0, \; [h_{i}, e_{j}] = c_{ij} e_{j}, \; [h_{i}, f_{j}] = - c_{ij} f_{j}, \; [e_{i}, f_{j}] = \delta_{ij} h_{i}
	\end{equation}
	\item and the standart Serre relations
	\item[LS2] for $i, j \in I$
	\begin{equation}
		\label{eq:scstr1}
		[e_{i}, e_{i}] = [f_{i}, f_{i}] = 0, \; \text{for} \; c_{ii} = 0,
	\end{equation}
	\begin{equation}
		\label{eq:scstr2}
		[e_{i}, e_{j}] = [f_{i}, f_{j}] = 0, \; \text{for} \; i \ne j, \; c_{ij} = 0,
	\end{equation}
	\begin{equation}
		\label{eq:scstr2a}
		[ e_{i} , [ e_{i}, e_{j} ] ] = [ f_{i} , [ f_{i}, f_{j} ] ] = 0, \; \text{for} \; c_{ii} \neq 0, \; i \neq m+n, \; j \in \{i-1, i+1\},
	\end{equation}
	\begin{equation}
		\label{eq:scstr2b}
		[ e_{m+n} , [ e_{m+n} , [ e_{m+n}, e_{m+n-1} ] ] ] = [ f_{m+n} , [ f_{m+n} , [ f_{m+n}, f_{m+n-1} ] ] ] = 0,
	\end{equation}
	\item plus higher order Serre relations for type \hyperref[fig:dynkin-diagrams-type1]{$\romannumeral1$}, \hyperref[fig:dynkin-diagrams-type2a]{$\romannumeral2a$} and \hyperref[fig:dynkin-diagrams-type2b]{$\romannumeral2b$} vertices
	\item[LS3] for $j-1, j, j+1 \in I$
	\begin{align}
	\label{eq:scstr3}
	[ [ [e_{j+1}, e_{j}], e_{j-1} ] , e_{j} ] = [ [ [f_{j+1}, f_{j}], f_{j-1} ] , f_{j} ] = 0.
	\end{align}	
	\end{description}
\end{proposition}

\begin{remark}
	Below $\times$ dots represent either white dots associated to even roots or grey dots associated to isotropic odd roots. Black dot corresponds to non-isotropic odd root.
	
	\begin{center}
		\begin{minipage}{0.5\textwidth}
			\centering	
			\begin{tikzpicture}
				\tikzset{/Dynkin diagram, root radius=.17cm}
				\dynkin[text style/.style={scale=1.0},labels*={j-1,j,j+1},edge length=1.7cm]A{xtx}
			\end{tikzpicture}
			\captionsetup{type=figure}
			\captionof{figure}{type $\romannumeral1$}
			\label{fig:dynkin-diagrams-type1}
		\end{minipage}%
		\begin{minipage}{0.5\textwidth}
			\centering
			\begin{tikzpicture}
				\tikzset{/Dynkin diagram, root radius=.17cm}
				\dynkin[text style/.style={scale=1.0},labels*={j-1,j,j+1},edge length=1.7cm]B{xto}
			\end{tikzpicture}
			\captionsetup{type=figure}
			\captionof{figure}{type $\romannumeral2a$}
			\label{fig:dynkin-diagrams-type2a}
		\end{minipage}%
		
		\vspace{2em}
		
		\begin{minipage}{0.5\textwidth}
			\centering
			\begin{tikzpicture}
				\tikzset{/Dynkin diagram, root radius=.17cm}
				\dynkin[text style/.style={scale=1.0},labels*={j-1,j,j+1},edge length=1.7cm]B{xt*}
			\end{tikzpicture}
			\captionsetup{type=figure}
			\captionof{figure}{type $\romannumeral2b$}
			\label{fig:dynkin-diagrams-type2b}
		\end{minipage}%
	\end{center}
	
\end{remark}

\label{rm:lsr}
\begin{remark}
	\leavevmode
	\begin{enumerate}		
		\item Relations \eqref{eq:scstr1} are also valid for $i \notin \tau$ by \eqref{eq:LS}.
		\item Relations \eqref{eq:scstr2a} are also valid for $c_{ii} = 0$, $i \neq m+n$, and $j \in \{i-1, i+1\}$, but in that case, they already follow from \cref{eq:LS,eq:sje,eq:scstr1}.
		\item Relations \eqref{eq:scstr3} are also valid for $c_{jj} \neq 0$ (in that case $|e_{j}| = |f_{j}| = \bar{0}$), but in that case, they already follow from \cref{eq:LS,eq:sje,eq:scstr1,eq:scstr2} (see \cite[Lemma 6.1.1]{Y94}).
	\end{enumerate}
\end{remark}
To find out how look like all Dynkin diagrams and simple root systems for all possible Borel subalgebras of $\mathfrak{g}$ see for example \cite{Z11}.

By assumption $\mathcal{H} = \langle h_{i} \rangle_{i \in I}$. We can construct an orthonormal basis $\{ h^{(k)} \}_{k \in I}$ of $\mathcal{H}$ with respect to the $\langle . , . \rangle$, i. e. $\langle h^{(i)}, h^{(j)} \rangle = \delta_{ij}$ for all $i,j \in I$. Then the Casimir operator $\Omega \in \mathfrak{g} \otimes \mathfrak{g}$ is given by the formula
\begin{align}
	\nonumber
	\Omega = \sum_{k \in I} h^{(k)} \otimes h^{(k)} + \sum_{\alpha \in \Delta^{+}_{\bar{0}}} ( e_{\alpha} \otimes f_{\alpha} + f_{\alpha} \otimes e_{\alpha} ) +
	\sum_{\alpha \in \Delta^{+}_{\bar{1}}} ( f_{\alpha} \otimes e_{\alpha} - e_{\alpha} \otimes f_{\alpha} ) =	\\
	\label{eq:CE}
	\sum_{k \in I} h^{(k)} \otimes h^{(k)} + \sum_{ \alpha \in \Delta^{+}} (-1)^{|e_{\alpha}|} e_{\alpha} \otimes f_{\alpha} +  \sum_{ \alpha \in \Delta^{+}} f_{\alpha} \otimes e_{\alpha}.
\end{align}
The element $\Omega$ is the even, invariant, supersymmetric, i. e.
\begin{enumerate}
	\item $|\Omega| = 0$;
	\item $[ x \otimes 1 + 1 \otimes x, \Omega ] = 0$ for all $x \in \mathfrak{g}$;
	\item $\Omega = \tau_{\mathfrak{g},\mathfrak{g}}(\Omega)$ where $\tau_{\mathfrak{g},\mathfrak{g}}$ is defined by \eqref{eq:taudef}.
\end{enumerate}
Note that
\[ \sum_{k \in I} \langle h^{(k)} , h_{i} \rangle h^{(k)} = h_{i}, \]
for all $i \in I$.

\vspace{10pt}

\section{Drinfeld super Yangians}

\label{sec:DSY}

\vspace{10pt}

In this section, we present the definition of the Drinfeld super Yangian associated with the Lie superalgebra $B(m,n)$ relative to all possible Borel subalgebras. Furthermore, we develop a concise framework for expressing this Drinfeld super Yangian in a minimalistic manner.

\vspace{10pt}

\subsection{Definition of Drinfeld super Yangian}
\label{sect:DSY}

\vspace{10pt}

We use the notations introduced in Subsection \ref{subsc:bs}. Following \cite{T19a,T19b} (see also \cite{D88}, \cite{S94} and \cite{G07}), define the Drinfeld super Yangian of $\mathfrak{g}$, denoted by $Y_{\hbar}(\mathfrak{g})$, to be the unital, associative $\Bbbk[[\hbar]]$-Hopf superalgebra. It is generated by $\{ h_{i,r} , x_{i,r}^{\pm} \}^{r \in \mathbb{N}_{0}}_{i \in I}$. The $\mathbb{Z}_2$-grading of these generators is specified as follows $|h_{i,r}| = \bar{0}$, $|x_{i,r}^{\pm}| = |i|$ for all $r \in \mathbb{N}_{0}$, $i \in I$. $Y_{\hbar}(\mathfrak{g})$ is subject to the following defining relations
\begin{description}
	\item[Y1] for all $i,j \in I$ and $r,s \in \mathbb{N}_{0}$
	\begin{equation}
		\label{eq:Cer}
		[h_{i,r}, h_{j,s}] = 0,
	\end{equation}
	\item[Y2] for all $i,j \in I$ and $s \in \mathbb{N}_{0}$
	\begin{equation}
		\label{eq:SYrc1}
		[h_{i,0}, x_{j,s}^{\pm}] = \pm c_{ij} x_{j,s}^{\pm},
	\end{equation}
	\item[Y3] for $i,j \in I$ and $r,s \in \mathbb{N}_{0}$
	\begin{equation}
		\label{eq:SYrc3}
		[x_{i,r}^{+}, x_{j,s}^{-}] = \delta_{i,j} h_{i,r+s},
	\end{equation}
	\item[Y4] for $i,j \in I$ and $r,s \in \mathbb{N}_{0}$; if $i=j$, then $|i| = \bar{0}$ or ($|i|=\bar{1}$, $c_{ii} \ne 0$ and $r = 0$)
	\begin{equation}
		\label{eq:SYrc2}
		[ h_{i,r+1} , x_{j,s}^{\pm} ] - [h_{i,r} , x_{j,s+1}^{\pm}] = \pm \frac{c_{ij} \hbar}{2} \{h_{i,r}, x_{j,s}^{\pm}\},
	\end{equation}
	\item[Y5] for $i,j \in I$ and $r,s \in \mathbb{N}_{0}$
	\begin{equation}
		\label{eq:SYrc4}
		[x_{i,r+1}^{\pm}, x_{j,s}^{\pm}] - [x_{i,r}^{\pm}, x_{j,s+1}^{\pm}] = \pm \frac{c_{ij} \hbar}{2} \{ x_{i,r}^{\pm} , x_{j,s}^{\pm} \}, \text{ unless } i=j \text{ and } |i| = \bar{1},
	\end{equation}	
	\item[Y6] for $i \in I$ and $r,s \in \mathbb{N}_{0}$
	\begin{equation}
		\label{eq:SYrc5}
		[h_{i,r}, x_{i,s}^{\pm}] = 0 \text{ if } c_{ii} = 0,
	\end{equation}
	\item[Y7] for $i,j \in I$ and $r,s \in \mathbb{N}_{0}$
	\begin{equation}
		\label{eq:bssr}
		[x_{i,r}^{\pm}, x_{j,s}^{\pm}] = 0 \text{ if } c_{ij} = 0,
	\end{equation}
	\item as well as cubic super Serre relations
	\item[Y8] for $i,j \in I$ ($i \neq m+n$ in case of \hyperref[fig:dynkin-diagrams-type2b]{$\romannumeral2b$} vertices), $j \in \{i-1, i+1\}$ and $r,s,t \in \mathbb{N}_{0}$
	\begin{equation}
		\label{eq:cssr}
		[ x_{i,r}^{\pm} , [ x_{i,s}^{\pm} , x_{j,t}^{\pm} ] ] + [ x_{i,s}^{\pm} , [x_{i,r}^{\pm} , x_{j,t}^{\pm}] ] = 0 \text{ if } c_{ii} \neq 0,
	\end{equation}
	\item[Y9] in case of \hyperref[fig:dynkin-diagrams-type2b]{$\romannumeral2b$} vertices for $r_{1},r_{2},r_{3},t \in \mathbb{N}_{0}$ and $S_{3}$ - symmetric group on the set $\{1,2,3\}$
	\begin{equation}
		\label{eq:cssra}
		\sum_{\sigma \in S_{3}} [ x_{m+n,r_{\sigma(1)}}^{\pm} , [ x_{m+n,r_{\sigma(2)}}^{\pm} , [ x_{m+n,r_{\sigma(3)}}^{\pm} , x_{m+n-1,t}^{\pm} ] ] ] = 0,
	\end{equation}
	\item and quartic super Serre relations for type \hyperref[fig:dynkin-diagrams-type1]{$\romannumeral1$}, \hyperref[fig:dynkin-diagrams-type2a]{$\romannumeral2a$} and \hyperref[fig:dynkin-diagrams-type2b]{$\romannumeral2b$} vertices
	\item[Y10] for $j-1,j,j+1 \in I$ and $r,s \in \mathbb{N}_{0}$
	\begin{equation}
		\label{eq:qssr}
		[ [ x_{j-1,r}^{\pm} , x_{j,0}^{\pm} ] , [ x_{j,0}^{\pm}, x_{j+1,s}^{\pm} ] ] = 0.
	\end{equation}
\end{description}

\begin{remark}
	\begin{enumerate}
		\item Similar to Remark \ref{rm:lsr}, cubic super Serre relations \eqref{eq:cssr} also hold for $c_{ii} = 0$, $i \neq m+n$, and $j \in \{i-1, i+1\}$, but in that case they already follow from \cref{eq:sje,eq:bssr}.
		\item Similar to Remark \ref{rm:lsr}, quartic super Serre relations \eqref{eq:qssr} also hold for for $c_{jj} \neq 0$, but in that case they already follow from \cref{eq:sje,eq:bssr,eq:cssr}.
		\item Generalizing the quartic super Serre relations \eqref{eq:qssr}, the following relations also hold \cite{T19b}:
		\[ [ [x_{j-1,r}^{\pm}, x_{j,k}^{\pm}], [x_{j,l}^{\pm}, x_{j+1,s}^{\pm}]] + [ [x_{j-1,r}^{\pm}, x_{j,l}^{\pm}], [x_{j,k}^{\pm},x_{j+1,s}^{\pm}]] = 0. \]
		This, in turn, using \eqref{eq:cssr} can be rewritten as
		\[ [[[x^{\pm}_{j-1,r}, x^{\pm}_{j,k}] , x^{\pm}_{j+1,s}], x^{\pm}_{j,l}] + [[[x^{\pm}_{j-1,r}, x^{\pm}_{j,l}] , x^{\pm}_{j+1,s}], x^{\pm}_{j,k}] = 0. \]
	\end{enumerate}
	\label{rm:yangrelrew}
\end{remark}

We note that the universal enveloping superalgebra $U(\mathfrak{g})$ is naturally embedded in $Y_{\hbar}(\mathfrak{g})$ as Hopf superalgebra, and the embedding is given by the formulas $h_{i} \to h_{i,0}$, $e_{i} \to x_{i,0}^{+}$, $f_{i} \to x_{i,0}^{-}$. We shall identify the universal enveloping superalgebra $U(\mathfrak{g})$ with its image in the Drinfeld super Yangian $Y_{\hbar}(\mathfrak{g})$.

Let $Y_{\hbar}^{0}(\mathfrak{g})$ be the subalgebra generated by the elements $\{ h_{i,r} \}_{i \in I, r \in \mathbb{N}_{0}}$.

\vspace{10pt}

\subsection{Minimalistic presentation for Drinfeld super Yangian}

\label{sb:mpsy}

\vspace{10pt}

To establish the Hopf superalgebra properties of a Drinfeld super Yangian, we introduce a more convenient representation for the superalgebras $Y_{\hbar}(\mathfrak{g})$. Such work was done in \cite{S94} for $Y_{\hbar}(\mathfrak{g})$ associated only with the distinguished Dynkin diagram (the result is stated without proof), and for non-super case in \cite{GNW18} and \cite{L93}.

From defining relations we can see that $Y_{\hbar}(\mathfrak{g})$ is generated by elements $\{ h_{ir} , x_{i0}^{\pm} \}^{r \in \{0,1\}}_{i \in I}$. We use equations \eqref{eq:SYrc1}, \eqref{eq:SYrc3} and \eqref{eq:SYrc2} to get recurrent formulas
\begin{equation}
	\label{eq:rec1}
	x_{i,r+1}^{\pm} = \pm (c_{ij})^{-1} [ h_{j1} - \frac{\hbar}{2} h_{j0}^2 , x_{ir}^{\pm} ],
\end{equation}
where $r \in \mathbb{N}_{0}$; if $c_{ii} \ne 0$, then $j=i$, and if $c_{ii} = 0$, then $j=i+1$ ($i \in I$);
\begin{equation}
	\label{eq:rec2}
	h_{ir} = [ x_{ir}^{+}, x_{i0}^{-} ],
\end{equation}
where $r \ge 2$ and $i \in I$.
We introduce the auxiliary generators for $i \in I$ by setting
\begin{equation}
	\widetilde{h}_{i 1} \overset{\operatorname{def}}{=} h_{i1} - \frac{\hbar}{2} h_{i0}^2.
\end{equation}
The $\mathbb{Z}_2$-grading of these generators is specified as follows $|\widetilde{h}_{i 1}|=|h_{i,r}| = \bar{0}$, $|x_{i,r}^{\pm}| = |i|$ for all $r \in \mathbb{N}_{0}$ and $i \in I$.

\begin{theorem}
	\label{th:mpy}
	 $Y_{\hbar}(\mathfrak{g})$ is isomorphic to the superalgebra generated by $\{ h_{ir} , x_{ir}^{\pm} \}^{r \in \{0,1\}}_{i \in I}$ subject only to the relations
	 \begin{description}
	 	\item[MY1] for all $i, j \in I$ and $0 \le r,s \le 1$
 		\begin{equation}
	 		\label{eq:mpy1}
	 		[h_{ir}, h_{js}] = 0,
	 	\end{equation}
 		\item[MY2] for all $i, j \in I$ and $0 \le s \le 1$
 		\begin{equation}
 		\label{eq:mpy2}
 		[h_{i0}, x_{js}^{\pm}] = \pm c_{ij} x_{js}^{\pm},
 		\end{equation}
	 	\item[MY3] for $i, j \in I$ and $0 \le r+s \le 1$
 	\begin{equation}
 		\label{eq:mpy3}
 		[x_{ir}^{+}, x_{js}^{-}] = \delta_{ij} h_{i,r+s},
 	\end{equation}
		\item[MY4] for $i, j \in I$
 	\begin{equation}
 		\label{eq:mpy4}
 		[\widetilde{h}_{i 1}, x_{j0}^{\pm}] = \pm c_{ij} x_{j1}^{\pm}, \text{ unless } i=j \text{ and } c_{ii} = 0,
 	\end{equation}
 	\item[MY5] for $i, j \in I$
 	\begin{equation}
 		\label{eq:mpy5}
 		[x_{i1}^{\pm}, x_{j0}^{\pm}] - [x_{i0}^{\pm}, x_{j1}^{\pm}] = \pm \frac{c_{ij} \hbar}{2} \{ x_{i0}^{\pm} , x_{j0}^{\pm} \}, \text{ unless } i=j \text{ and } |i| = \bar{1},
 	\end{equation}
 	\item[MY6] for $i \in I$ and $0 \le s \le 1$
 	\begin{equation}
 		\label{eq:mpy6}
 		[h_{i1}, x_{is}^{\pm}] = 0, \text{ if } c_{ii}=0,
 	\end{equation}
	\item[MY7] for $i,j \in I$
 	\begin{equation}
 		\label{eq:mpy7}
 		[x_{i0}^{\pm}, x_{j0}^{\pm}] = 0, \text{ if } c_{ij} = 0,
 	\end{equation}
 	\item[MY8] for $i \in I$ ($i \neq m+n$ in case of \hyperref[fig:dynkin-diagrams-type2b]{$\romannumeral2b$} vertices), $j \in \{i-1, i+1\}$
 	\begin{equation}
 		\label{eq:mpy8}
 		[ x_{i0}^{\pm} , [ x_{i0}^{\pm} , x_{j0}^{\pm} ] ] = 0,  \text{ if } c_{ii} \neq 0,
 	\end{equation}
	\item[MY9] in case of \hyperref[fig:dynkin-diagrams-type2b]{$\romannumeral2b$} vertices
	\begin{equation}
		\label{eq:mpy8a}
		[ x_{m+n,0}^{\pm} , [ x_{m+n,0}^{\pm} , [ x_{m+n,0}^{\pm}, x_{m+n-1,0}^{\pm} ] ] ] = 0,
	\end{equation}
	\item[MY10] for $j-1,j,j+1 \in I$ and for type \hyperref[fig:dynkin-diagrams-type1]{$\romannumeral1$}, \hyperref[fig:dynkin-diagrams-type2a]{$\romannumeral2a$} and \hyperref[fig:dynkin-diagrams-type2b]{$\romannumeral2b$} vertices
 	\begin{equation}
 		\label{eq:mpy9}
 		[ [ x_{j-1,0}^{\pm} , x_{j0}^{\pm} ] , [ x_{j0}^{\pm}, x_{j+1,0}^{\pm} ] ] = 0.
 	\end{equation}
	 \end{description}
\end{theorem}
\begin{proof}
	The proof of the theorem is established through a sequence of logical steps based on various remarks and lemmas. We summarize these key steps as follows:
	\begin{enumerate}
		\item The equality \eqref{eq:Cer} is derived from Lemmas \ref{lm:cartelrelii} and \ref{lm:cartelrelinej}.
		
		\item The relation \eqref{eq:SYrc1} is established with reference to Lemma \ref{lm:h0eq}.
		
		\item The equality \eqref{eq:SYrc3} is established with support from Lemmas \ref{lm:SYrc34} and \ref{lm:cartelrelii}.
		
		\item The relation \eqref{eq:SYrc2} is deduced using Lemma \ref{lm:hixi}.
		
		\item The relation \eqref{eq:SYrc4} is derived by means of Lemma \ref{lm:xijxijii}.
		
		\item The relation \eqref{eq:SYrc5} is obtained through the insights provided by Lemma \ref{lm:SYrc5}.
		
		\item The relations \eqref{eq:bssr} and \eqref{eq:cssr} are established based on the analysis presented in Lemma \ref{lm:bcssr}.
		
		\item The relation \eqref{eq:cssra} is derived following the findings in Lemma \ref{lm:cssra}.
		
		\item The relation \eqref{eq:qssr} is logically connected to the results of Lemma \ref{lm:leqmp}.
		
		\item Finally, the relation \eqref{rm:lsar} is established, drawing upon the insights derived from Lemma \ref{lm:lsar}.
	\end{enumerate}

	By systematically referencing these remarks and lemmas, the theorem's proof is constructed, providing a clear and organized line of reasoning.
\end{proof}
In this superalgebra, we also define elements $x_{ir}^{\pm}$ ($r \ge 2$) and $h_{ir}$ ($r \ge 2$) for $i \in I$ using \eqref{eq:rec1} and \eqref{eq:rec2}. The $\mathbb{Z}_2$-grading of these generators is specified as in $Y_{\hbar}(\mathfrak{g})$.

\begin{remark}
	\label{rm:addrel}
	As noted in \cite{GNW18} in order to prove Theorem \ref{th:mpy} we must prove that equation
	\begin{equation}
		\label{rm:lsar}
		[ [\widetilde{h}_{j1},x_{i1}^{+}] , x_{i1}^{-} ] + [ x_{i1}^{+} , [ \widetilde{h}_{j1},x_{i1}^{-} ] ] = 0
	\end{equation}
	or
    \begin{equation}
   	\label{eq:hia12}
    [h_{j1}, h_{i2}] = 0
	\end{equation}
  	can be deduced from relations \eqref{eq:mpy1} - \eqref{eq:mpy9}, where if $c_{ii} \ne 0$, then $j=i$, and if $c_{ii} = 0$, then $j=i+1$ ($i \in I$). Note that it follows from Lemmas \eqref{lm:l2} and \eqref{lm:cel02} that \eqref{rm:lsar} is equivalent to \eqref{eq:hia12}.
\end{remark}

\vspace{10pt}

\section{Hopf superalgebra structure on Drinfeld super Yangians}

\vspace{10pt}

\label{sec:hssdsy}

In this Section we explicitly describe a Hopf superalgebra structure on a Drinfeld super Yangian of Lie superalgebra $B(m,n)$. We use notations introduced in Subsection \ref{sb:mpsy}.
\begin{theorem}
	\label{th:hssdsy}
	$Y_{\hbar}(\mathfrak{g})$ is a Hopf superalgebra. 
	
	\textbf{Counit:} the counit $\epsilon: Y_{\hbar}(\mathfrak{g}) \to \Bbbk$ is defined by the following equations:
	\begin{description}
		\item[Co1]
		\begin{equation}
			\epsilon(1) = 1,
		\end{equation}
		\item[Co2] for $x \in \{ h_{ir} , x_{ir}^{\pm} \}^{r \in \mathbb{N}_{0}}_{i \in I}$
		\begin{equation}
			\epsilon(x) = 0.
		\end{equation}
	\end{description}			

	\textbf{Comultiplication:} The comultiplication $\Delta: Y_{\hbar}(\mathfrak{g}) \to Y_{\hbar}(\mathfrak{g}) \otimes Y_{\hbar}(\mathfrak{g})$ is given by the following relations:
	\begin{description}
		\item[Com1] for all $x \in \mathfrak{g}$
		\begin{equation}
			\Delta(x) = \square(x),
		\end{equation}
		\item[Com2]
			\begin{align}
			\Delta(h_{i1}) &= \square(h_{i1}) + \hbar(  h_{i0} \otimes h_{i0} + [ h_{i0} \otimes 1, \Omega^{+}] )  \nonumber \\
			&=	\square(h_{i1}) + \hbar( h_{i0} \otimes h_{i0} - \sum_{ \alpha \in \Delta^{+}} (\alpha_{i}, \alpha) x_{\alpha}^{-} \otimes x_{\alpha}^{+} ).
		\end{align}
		\item[Com3] for all $r \in \mathbb{N}_{0}$; if $c_{ii} \ne 0$, then $j=i$, and if $c_{ii} = 0$, then $j=i+1$ ($i \in I$)
		\begin{equation}
			\Delta(x_{i,r+1}^{\pm}) = \pm (c_{ij})^{-1} [ \Delta(h_{j1}) - \frac{\hbar}{2} \Delta^2(h_{j0}) , \Delta(x_{ir}^{\pm}) ],
		\end{equation}
		\item[Com4] for all $i \in I$ and $r \ge 2$
		\begin{equation}
			\Delta(h_{ir}) = [ \Delta(x_{ir}^{+}), \Delta(x_{i0}^{-}) ],
		\end{equation}
	\end{description}
	
	\textbf{Antipode:} The antipode $S: Y_{\hbar}(\mathfrak{g}) \to Y_{\hbar}(\mathfrak{g})^{op \; cop}$ is described by the following equations:
	\begin{description}
		\item[Ant1] for all $x \in \mathfrak{g}$
		\begin{equation}
			\label{eq:antipodeY1}
			S(x) = -x,
		\end{equation}
		\item[Ant2]
		\begin{equation}
			S(h_{i1}) =	-h_{i1} + \hbar ( h_{i0}^2 + \sum_{ \alpha \in \Delta^{+}} (-1)^{1+|\alpha|} (\alpha_{i}, \alpha) x_{\alpha}^{-} x_{\alpha}^{+} ).
		\end{equation}
		\item[Ant3] for all $r \in \mathbb{N}_{0}$; if $c_{ii} \ne 0$, then $j=i$, and if $c_{ii} = 0$, then $j=i+1$ ($i \in I$)
		\begin{equation}
			S(x_{i,r+1}^{\pm}) = \mp (c_{ij})^{-1} [ S(h_{j1}) - \frac{\hbar}{2} S^2(h_{j0}) , S(x_{ir}^{\pm}) ],
		\end{equation}
		\item[Ant4] for all $i \in I$ and $r \ge 2$
		\begin{equation}
			\label{eq:antipodeY2}
			S(h_{ir}) = - [ S(x_{ir}^{+}), S(x_{i0}^{-}) ].
		\end{equation}
	\end{description}		
\end{theorem}
\begin{proof}
	By combining the insights from the Subsection \ref{subsc:bs} with the results of Theorem \ref{th:mpy}, we can observe that the proof aligns explicitly with the one presented in \cite[Section 4]{MS22}.
\end{proof}

\vspace{10pt}

\section{Proofs of auxiliary results}

\label{sec:par}

\vspace{10pt}

The proof of Theorem \ref{th:mpy}  is splitted in several lemmas and propositions. We give all necessary proofs in this section.

\vspace{10pt}

\subsection{General results about Drinfeld super Yangians}

\label{sub:grdsy}

\vspace{10pt}

Set for any $i \in I$
\begin{equation}
	\label{eq:hln}
	\widetilde{h}_{i}(t) := \hbar \sum_{r \ge 0} \widetilde{h}_{i,r} t^{-r-1} = \log ( 1 + \hbar \sum_{r \ge 0} h_{i,r} t^{-r-1} ) \in Y_{\hbar}(\mathfrak{g})^{0}[[t^{-1}]].
\end{equation}

\begin{lemma}
	\label{lm:thfdpr}
	Let \eqref{eq:Cer} hold for $i,j \in I$ and $0 \le r,s \le v$, let \eqref{eq:SYrc2} hold for $i,j \in I$, $0 \le r \le v$ and $s \in \mathbb{N}_{0}$, and let \eqref{eq:SYrc5} hold for $i \in I$ ($c_{ii}=0$), $0 \le r,s \le v$. Then for $i,j \in I$, $0 \le r \le v$, $s \in \mathbb{N}_{0}$,
	\begin{equation}
		\label{eq:thx}
		[ \widetilde{h}_{i,r}, x_{j,s}^{\pm} ] = \pm c_{ij} x_{j,r+s}^{\pm} \pm c_{ij} \sum_{p=1}^{\lfloor r/2 \rfloor} \binom{r}{2p} \frac{(\hbar c_{ij} / 2)^{2p}}{2p+1} x_{j,r+s-2p}^{\pm}.
	\end{equation}
\end{lemma}
\begin{proof}
	Note that from \eqref{eq:hln} follows that for arbitrary $r \in \mathbb{N}_{0}$ we have $\widetilde{h}_{i,r} = f(h_{i,0}, h_{i,1} , ..., h_{i,r})$ for some element of free algebra $f \in \Bbbk \langle x_{1},...,x_{r+1} \rangle$. Hence, while deriving \eqref{eq:thx}, we may and shall assume that \eqref{eq:Cer} and \eqref{eq:SYrc2} hold for all $r,s \in \mathbb{N}_{0}$. Then the result follows from \cite[Lemma 2.7, Lemma 2.9, Remark 3.1]{GT13}.
\end{proof}

Set $\dbtilde{h}_{ij,0} = h_{i,0}$ ($i,j \in I$), and define inductively for $r \in \mathbb{N}$
\[ \dbtilde{h}_{ij,r} = \widetilde{h}_{i,r} - \sum_{p=1}^{\lfloor r/2 \rfloor} \binom{r}{2p} \frac{(\hbar (\alpha_{i},\alpha_{j}) / 2)^{2p}}{2p+1} \dbtilde{h}_{ij,r-2p}. \]
We have

\begin{lemma}
	\label{lm:hwdpr}
	Suppose that Lemma \ref{lm:thfdpr} holds. Then in the same notations
	\begin{equation}
		\label{eq:thxo}
		[ \dbtilde{h}_{ij,r} , x_{j,s}^{\pm} ] = \pm (\alpha_{i}, \alpha_{j})x_{j,r+s}^{\pm},
	\end{equation}
	for $i, j \in I$, $0 \le r \le v$, $s \in \mathbb{N}_{0}$, and
	\begin{equation}
		\label{eq:thxs}
		\dbtilde{h}_{ij,r} = h_{i,r} + f(h_{i,0}, h_{i,1}, ... , h_{i,r-1})
	\end{equation}
	for some polynomial $f(x_{1},x_{2},...,x_{r}) \in \Bbbk[x_{1},x_{2},...,x_{r}]$.
\end{lemma}
\begin{proof}
	The proof is the same as in \cite[Lemma 3.16]{MS22}. It is based on Lemma \ref{lm:thfdpr}.
\end{proof}

\begin{lemma}
	\label{lm:ijtrancart}
	Let $p, n, m \in I$ and $z \in \mathbb{N}_{0}$. Suppose that $[h_{pz} , h_{nv}] = 0$, $[h_{pz}, h_{mv}] = 0$ for $0 \le v \le s^{'}-1$ $(s^{'} \in \mathbb{N}_{0})$; $(\alpha_{p}, \alpha_{n}) \ne 0$ and $(\alpha_{p}, \alpha_{m}) \ne 0$. Moreover, let $[h_{nv_1} , h_{n v_2}] = 0$ and $[h_{mv_1} , h_{m v_2}] = 0$ for $0 \le v_1, v_2 \le s^{'}$ and \eqref{eq:SYrc2} hold for $(i,j)=(n,p)=(m,p)$, $0 \le r \le s^{'}$ and any $s \in \mathbb{N}_{0}$. Then $[h_{pz}, h_{ns^{'}}] = \frac{(\alpha_{n} , \alpha_{p})}{(\alpha_{m} , \alpha_{p})} [h_{pz}, h_{ms^{'}}]$.
\end{lemma}
\begin{proof}
	The proof is the same as in \cite[Lemma 3.17]{MS22}. It is based on Lemmas \ref{lm:thfdpr} and \ref{lm:hwdpr}.
\end{proof}

\vspace{10pt}

\subsection{Proof of minimalistic presentation theorem}

\label{sub:pmpt}

\vspace{10pt}

\begin{lemma}
	\label{lm:h0eq}
	The relation \eqref{eq:SYrc1} is satisfied for all $i,j \in I$ and $r \in \mathbb{N}_{0}$. Moreover, for the same parameters
	\begin{equation}
		\label{eq:h1eq}
		[\widetilde{h}_{i1} , x_{jr}^{\pm}] = \pm c_{ij} x_{j,r+1}^{\pm}.
	\end{equation}
\end{lemma}
\begin{proof}
	The proof is the same as in \cite[Lemma 3.1]{MS22}.
\end{proof}

\begin{lemma}
	\label{lm:l2}
	The relation \eqref{eq:SYrc3} holds when $i=j$, $0 \le r +s \le 2$.
\end{lemma}
\begin{proof}
	From \eqref{eq:mpy1}, \eqref{eq:mpy3} and Lemma \ref{lm:h0eq} for all $i \in I$ ($j \in I$ such that $c_{ij} \ne 0$), we have
	\[ 0 = [ h_{i1}, \widetilde{h}_{j1} ] = [ [x_{i1}^{+}, x_{i0}^{-}] , \widetilde{h}_{j1} ] = c_{ij} ( [ x_{i1}^{+}, x_{i1}^{-} ] - [ x_{i2}^{+} , x_{i0}^{-} ] ). \]
	On the other hand,
	\[ 0 = [ h_{i1}, \widetilde{h}_{j1} ] = [ [x_{i0}^{+}, x_{i1}^{-}] , \widetilde{h}_{j1} ] = c_{ij} ( [ x_{i0}^{+}, x_{i2}^{-} ] - [ x_{i1}^{+} , x_{i1}^{-} ] ). \]
	Therefore
	\[ [ x_{i0}^{+}, x_{i2}^{-} ] = [ x_{i1}^{+}, x_{i1}^{-} ] = [ x_{i2}^{+} , x_{i0}^{-} ] = h_{i2}. \]
\end{proof}

\begin{lemma}
	\label{lm:hx20}
	The relation \eqref{eq:SYrc2} holds when $i=j$, $(r,s) =(1,0)$, and $|i| = \bar{0}$, i.e.
	\[ [ h_{i2} , x_{i0}^{\pm} ] - [h_{i1} , x_{i1}^{\pm}] = \pm \frac{c_{ii} \hbar}{2} \{h_{i1}, x_{i0}^{\pm}\}, \]
\end{lemma}
\begin{proof}
	The proof is analogous to that in \cite[Lemma 2.26]{GNW18}.
\end{proof}

\begin{lemma}
	\label{lm:SYrc34}
	Let $i,j \in I$ and $i \ne j$. The relations \eqref{eq:SYrc3}, \eqref{eq:SYrc4}, and \eqref{eq:bssr} hold for any $r,s \in \mathbb{N}_{0}$.
\end{lemma}
\begin{proof}
	We prove \eqref{eq:SYrc4} by induction on $r$ and $s$. The initial case $r=s=0$ is our assumption \eqref{eq:mpy5}. Let $X^{\pm}(r,s)$ be the result of substracting the right-hand side of \eqref{eq:SYrc4} from the left-hand side. Suppose that $X^{\pm}(r,s)=0$. Note that if we apply $[\widetilde{h}_{m1} , \cdot]$ and $[\widetilde{h}_{n1} , \cdot]$ ($m,n \in I$) to $X^{\pm}(r,s)=0$ we get from \eqref{eq:h1eq}
	\[ 0 = (\alpha_{i}, \alpha_{m}) X^{\pm}(r+1,s) + (\alpha_{j}, \alpha_{m}) X^{\pm}(r,s+1), \]
	\[ 0 = (\alpha_{i}, \alpha_{n}) X^{\pm}(r+1,s) + (\alpha_{j}, \alpha_{n}) X^{\pm}(r,s+1). \]
	Consider the matrix $A = \bigl( \begin{smallmatrix}(\alpha_{i},\alpha_{m}) & (\alpha_{j},\alpha_{m})\\ (\alpha_{i},\alpha_{n}) & (\alpha_{j},\alpha_{n})\end{smallmatrix}\bigr)$. When the determinant of $A$ is non-zero, we have $X^{\pm}(r+1,s)=X^{\pm}(r,s+1)=0$.
	
	In order to determine when the determinant of the matrix $A$ is non-zero depending on the grading of roots it is sufficient to consider the following Dynkin (sub)diagrams.
	\begin{enumerate}
		\item Select in \hyperref[fig:dynkin-diagrams-01]{case 1.1}, \hyperref[fig:dynkin-diagrams-02]{case 1.2} and \hyperref[fig:dynkin-diagrams-03]{case 1.3} $m=i$, $n=j$ in order to get the nonzero determinant.
		\begin{center}
			\begin{minipage}{0.3\textwidth}
				\centering	
				\begin{tikzpicture}
					\tikzset{/Dynkin diagram, root radius=.17cm}
					\dynkin[text style/.style={scale=1.0},labels*={i,j},edge length=1.7cm]A{xx}
				\end{tikzpicture}
				\captionsetup{type=figure}
				\captionof{figure}{case 1.1}
				\label{fig:dynkin-diagrams-01}
			\end{minipage}%
			\begin{minipage}{0.3\textwidth}
				\centering
				\begin{tikzpicture}
					\tikzset{/Dynkin diagram, root radius=.17cm}
					\dynkin[text style/.style={scale=1.0},labels*={i,j},edge length=1.7cm]B{xo}
				\end{tikzpicture}
				\captionsetup{type=figure}
				\captionof{figure}{case 1.2}
				\label{fig:dynkin-diagrams-02}
			\end{minipage}%
			
			\vspace{2em}
			
			\begin{minipage}{0.3\textwidth}
				\centering
				\begin{tikzpicture}
					\tikzset{/Dynkin diagram, root radius=.17cm}
					\dynkin[text style/.style={scale=1.0},labels*={i,j},edge length=1.7cm]B{x*}
				\end{tikzpicture}
				\captionsetup{type=figure}
				\captionof{figure}{case 1.3}
				\label{fig:dynkin-diagrams-03}
			\end{minipage}%
			\vspace{2em}
		\end{center}
		
		\item Select in \hyperref[fig:dynkin-diagrams-1]{case 2.1}, \hyperref[fig:dynkin-diagrams-2]{case 2.2} and \hyperref[fig:dynkin-diagrams-3]{case 2.3}
		
		$|i|=|j|=\bar{0} \Rightarrow m=i, n=j$; \\
		$|i|=\bar{1}, |j|=\bar{0} \Rightarrow m=i^{'}, n=j$; \\
		$|i|=\bar{0}, |j|=\bar{1} \Rightarrow m=i, n=i^{'}$.
		
			\begin{center}
			\begin{minipage}{0.3\textwidth}
				\centering	
				\begin{tikzpicture}
					\tikzset{/Dynkin diagram, root radius=.17cm}
					\dynkin[text style/.style={scale=1.0},labels*={i,i^{'},j},edge length=1.7cm]A{xxx}
				\end{tikzpicture}
				\captionsetup{type=figure}
				\captionof{figure}{case 2.1}
				\label{fig:dynkin-diagrams-1}
			\end{minipage}%
			\begin{minipage}{0.3\textwidth}
				\centering
				\begin{tikzpicture}
					\tikzset{/Dynkin diagram, root radius=.17cm}
					\dynkin[text style/.style={scale=1.0},labels*={i,i^{'},j},edge length=1.7cm]B{xxo}
				\end{tikzpicture}
				\captionsetup{type=figure}
				\captionof{figure}{case 2.2}
				\label{fig:dynkin-diagrams-2}
			\end{minipage}%
			
			\vspace{2em}
			
			\begin{minipage}{0.3\textwidth}
				\centering
				\begin{tikzpicture}
					\tikzset{/Dynkin diagram, root radius=.17cm}
					\dynkin[text style/.style={scale=1.0},labels*={i,i^{'},j},edge length=1.7cm]B{xx*}
				\end{tikzpicture}
				\captionsetup{type=figure}
				\captionof{figure}{case 2.3}
				\label{fig:dynkin-diagrams-3}
			\end{minipage}%
		\end{center}
		
		For \hyperref[fig:dynkin-diagrams-3]{case 2.3} when $|i|=|j|=\bar{1} \Rightarrow m=i^{'}, n=j$.
		Note that \hyperref[fig:dynkin-diagrams-1]{case 2.1} can occur only when $|I| > 3$. Thus in \hyperref[fig:dynkin-diagrams-4]{case 2.4} or \hyperref[fig:dynkin-diagrams-5]{case 2.5} when $|i|=|j|=\bar{1} \Rightarrow m=i^{'}$, $n=j^{'}$.
		
		\begin{center}
			\begin{minipage}{0.4\textwidth}
				\centering	
				\begin{tikzpicture}
					\tikzset{/Dynkin diagram, root radius=.17cm}
					\dynkin[text style/.style={scale=1.0},labels*={i,i^{'},j,j^{'}},edge length=1.7cm]A{txtx}
				\end{tikzpicture}
				\captionsetup{type=figure}
				\captionof{figure}{case 2.4}
				\label{fig:dynkin-diagrams-4}
			\end{minipage}%
		
			\vspace{2em}
			
			\begin{minipage}{0.4\textwidth}
				\centering
				\begin{tikzpicture}
					\tikzset{/Dynkin diagram, root radius=.17cm}
					\dynkin[text style/.style={scale=1.0},labels*={i^{'},i,j^{'},j},edge length=1.7cm]A{xtxt}
				\end{tikzpicture}
				\captionsetup{type=figure}
				\captionof{figure}{case 2.5}
				\label{fig:dynkin-diagrams-5}
			\end{minipage}%
		\end{center}
			
		\vspace{2em}
		
		\item Select in \hyperref[fig:dynkin-diagrams-6]{case 3.1}, \hyperref[fig:dynkin-diagrams-1]{case 3.2} and \hyperref[fig:dynkin-diagrams-8]{case 3.3}
		
		$|i|=|j|=\bar{0} \Rightarrow m=i, n=j$;
		
		$|i|=\bar{1}, |j|=\bar{0} \Rightarrow m=i^{'}, n=j$;
		
		$|i|=\bar{0}, |j|=\bar{1} \Rightarrow m=i, n=j^{'}$;
		
		$|i|=|j|=\bar{1} \Rightarrow m=i^{'}, n=j^{'}$.
		
		\begin{center}
			\begin{minipage}{0.4\textwidth}
				\centering	
				\begin{tikzpicture}
					\tikzset{/Dynkin diagram, root radius=.17cm}
					\dynkin[text style/.style={scale=1.0},labels*={i,i^{'},j^{'},j},edge length=1.7cm]A{xxxx}
				\end{tikzpicture}
				\captionsetup{type=figure}
				\captionof{figure}{case 3.1}
				\label{fig:dynkin-diagrams-6}
			\end{minipage}%
			
			\vspace{2em}
			
			\begin{minipage}{0.4\textwidth}
				\centering	
				\begin{tikzpicture}
					\tikzset{/Dynkin diagram, root radius=.17cm}
					\dynkin[text style/.style={scale=1.0},labels*={i,i^{'},j^{'},j},edge length=1.7cm]B{xxxo}
				\end{tikzpicture}
				\captionsetup{type=figure}
				\captionof{figure}{case 3.2}
				\label{fig:dynkin-diagrams-7}
			\end{minipage}%
			
			\vspace{2em}
			
			\begin{minipage}{0.4\textwidth}
				\centering	
				\begin{tikzpicture}
					\tikzset{/Dynkin diagram, root radius=.17cm}
					\dynkin[text style/.style={scale=1.0},labels*={i,i^{'},j^{'},j},edge length=1.7cm]B{xxx*}
				\end{tikzpicture}
				\captionsetup{type=figure}
				\captionof{figure}{case 3.3}
				\label{fig:dynkin-diagrams-8}
			\end{minipage}%
			
			\vspace{2em}
		\end{center}
		
	\end{enumerate}
	The result follows by induction hypothesis.
	
	We prove \eqref{eq:SYrc3} by induction on $r$ and $s$. The initial case $r=s=0$ is our assumption \eqref{eq:mpy3}. Let $X^{\pm}(r,s)$ be the result of substracting the right-hand side of \eqref{eq:SYrc3} from the left-hand side. Suppose that $X^{\pm}(r,s)=0$. Note that if we apply $[\widetilde{h}_{m1} , \cdot]$ and $[\widetilde{h}_{n1} , \cdot]$ ($m,n \in I$) to $X^{\pm}(r,s)=0$ we get from \eqref{eq:h1eq} that
	\[ 0 = (\alpha_{i}, \alpha_{m}) X^{\pm}(r+1,s) + (-1) (\alpha_{j}, \alpha_{m}) X^{\pm}(r,s+1), \]
	\[ 0 = (\alpha_{i}, \alpha_{n}) X^{\pm}(r+1,s) + (-1) (\alpha_{j}, \alpha_{n}) X^{\pm}(r,s+1). \]
	Consider the matrix $A = \bigl( \begin{smallmatrix}(\alpha_{i},\alpha_{m}) & -(\alpha_{j},\alpha_{m})\\ (\alpha_{i},\alpha_{n}) & -(\alpha_{j},\alpha_{n})\end{smallmatrix}\bigr)$. When the determinant of $A$ is non-zero, we have $X^{\pm}(r+1,s)=X^{\pm}(r,s+1)=0$. We apply the same arguments as above to deduce that it is always possible to select $m$ and $n$ in such a way that $\det(A) \ne 0$. The result follows by induction hypothesis.
	
	We prove \eqref{eq:bssr} by induction on $r$ and $s$. Denote the left hand side of \eqref{eq:bssr} by $X^{\pm}(r,s)$. The initial case $r=s=0$ is our assumption \eqref{eq:mpy7}. Suppose that $X^{\pm}(r,s)=0$. We apply $[\widetilde{h}_{m1} , \cdot]$ and $[\widetilde{h}_{n1} , \cdot]$ to $X^{\pm}(r,s) = 0$ to get 
	\[ 0 = (\alpha_{i}, \alpha_{m}) X^{\pm}(r+1,s) + (\alpha_{j}, \alpha_{m}) X^{\pm}(r,s+1), \]
	\[ 0 = (\alpha_{i}, \alpha_{n}) X^{\pm}(r+1,s) + (\alpha_{j}, \alpha_{n}) X^{\pm}(r,s+1). \]
	Consider the matrix $A = \bigl( \begin{smallmatrix}(\alpha_{i},\alpha_{m}) & (\alpha_{j},\alpha_{m})\\ (\alpha_{i},\alpha_{n}) & (\alpha_{j},\alpha_{n})\end{smallmatrix}\bigr)$. When the determinant of $A$ is non-zero, we have $X^{\pm}(r+1,s)=X^{\pm}(r,s+1)=0$. We apply the same arguments as above to deduce that it is always possible to select $m$ and $n$ in such a way that $\det(A) \ne 0$. The result follows by induction hypothesis.
\end{proof}

\begin{lemma}
	\label{lm:hxinej}
	Suppose that $i,j \in I$ and $i \ne j$. The equation \eqref{eq:SYrc2} holds for any $r,s \in \mathbb{N}_{0}$.
\end{lemma}
\begin{proof}
	The proof is analogous to that in \cite[Lemma 3.13]{MS22}.
\end{proof}

\begin{lemma}
	\label{lm:cel02}
	For all $i,j \in I$ and $s \in \mathbb{N}_{0}$
	\[ [h_{i0}, h_{js}] = 0. \]
\end{lemma}
\begin{proof}
	By Lemma \ref{lm:l2} and \eqref{eq:rec2}
	\[ [ h_{i0} , h_{js} ] = [ h_{i0} , [x_{js}^{+} , x_{j0}^{-}] ] = (\alpha_{i} , \alpha_{j}) ( h_{js} - h_{js} ) = 0. \]
\end{proof}

\begin{lemma}
	\label{lm:h12oe}
	$[h_{j1}, h_{i2}] = 0$ for all $i,j \in I$.
\end{lemma}
\begin{proof}
	For any $j \in I$ we have by Lemmas \ref{lm:h0eq}, \ref{lm:l2} and relation \eqref{eq:mpy6}
	\begin{align}
		& [ h_{j1} , h_{i2} ] = [ h_{j1} , [ x_{i1}^{+}, x_{i1}^{-} ] ] = (-1) ( [ x_{i1}^{+}, [ x_{i1}^{-}, h_{j1} ] ] + (-1)^{|i|} [ x_{i1}^{-}, [ h_{j1}, x_{i1}^{+} ] ] )  = \nonumber \\
	 	& [ x_{i1}^{+}, [ h_{j1}, x_{i1}^{-} ] ] + [ [ h_{j1}, x_{i1}^{+} ], x_{i1}^{-} ] = \nonumber  \\
	&	[ x_{i1}^{+}, - c_{ij} x_{i2}^{-} - \frac{\hbar}{2} [ h_{j0}^2 , x_{i1}^{-} ] ] + [  c_{ij} x_{i2}^{+} + \frac{\hbar}{2} [ h_{j0}^2 , x_{i1}^{+} ], x_{i1}^{-} ] = \nonumber \\
	&	 c_{ij} ( [x_{i2}^{+} , x_{i1}^{-} ] - [ x_{i1}^{+}, x_{i2}^{-} ] ) + \frac{\hbar}{2} ( [[ h_{j0}^2 , x_{i1}^{+} ], x_{i1}^{-} ] - [ x_{i1}^{+}, [ h_{j0}^2 , x_{i1}^{-} ] ] ) = \nonumber \\
	&	c_{ij} ( [x_{i2}^{+} , x_{i1}^{-} ] - [ x_{i1}^{+}, x_{i2}^{-} ] ) + \frac{\hbar}{2} ( [[ h_{j0}^2 , x_{i1}^{+} ], x_{i1}^{-} ] + [ h_{j0}^2, h_{i2} ] - [[ h_{j0}^2 , x_{i1}^{+} ], x_{i1}^{-} ] ) = \nonumber \\
	& c_{ij} ( [x_{i2}^{+} , x_{i1}^{-} ] - [ x_{i1}^{+}, x_{i2}^{-} ] ). \label{eq:hi12f}
	\end{align}
	
	On the other hand, by Lemmas \ref{lm:l2}, \ref{lm:cel02}
	\[ [ h_{j1} , h_{i2} ] = [ \widetilde{h}_{j1} , h_{i2} ] = [ \widetilde{h}_{j1} , [x_{i1}^{+},x_{i1}^{-}] ] = c_{ij} ( [x_{i1}^{+},x_{i2}^{-}] - [x_{i2}^{+},x_{i1}^{-}] ). \]
	
	Thus we get
	\[ [ h_{j1} , h_{i2} ] = 0. \]
\end{proof}

\begin{lemma}
	\label{lm:lsar}
	The equation \eqref{rm:lsar} holds for all $i \in I$.
\end{lemma}
\begin{proof}
	The result follows from \eqref{eq:hia12} and Lemma \ref{lm:h12oe}.
\end{proof}

\begin{lemma}
	\label{lm:r233}
	Equations \eqref{eq:Cer}, \eqref{eq:SYrc3} $($for $0 \le r+s \le 3)$ and \eqref{eq:SYrc2} $($for $0 \le r \le 1; \; s \in \mathbb{N}_{0})$ hold for $i=j$ $(i \in I)$.
\end{lemma}
\begin{proof}
	Consider the equation \eqref{eq:Cer}. The proof follows from \eqref{eq:mpy1} and Lemmas \ref{lm:cel02}, \ref{lm:h12oe}.
	
	Consider the equation \eqref{eq:SYrc3}. It follows from Lemma \ref{lm:l2} for $0 \le r+s \le 2$. By Lemmas \eqref{lm:l2} and \eqref{lm:h12oe} we have
	\[ 0 = [ h_{i2} , \widetilde{h}_{i^{'}1} ] = [ [ x_{i2}^{+} , x_{i0}^{-} ] , \widetilde{h}_{i^{'}1} ] = ( \alpha_{i^{'}} , \alpha_{i} ) ( [ x_{i2}^{+} , x_{i1}^{-} ] - [ x_{i3}^{+} , x_{i0}^{-} ] ), \]
	where if $c_{ii} \ne 0$, then $i^{'}=i$; otherwise $i^{'}=i+1$.	Apply $[\widetilde{h}_{i^{'}1}, \cdot]$ to
	\[ [ x_{i2}^{+} , x_{i0}^{-} ] = [ x_{i1}^{+} , x_{i1}^{-} ] = [ x_{i0}^{+} , x_{i2}^{-} ]. \]
	Then
	\[ 0 = [ x_{i2}^{+} , x_{i1}^{-} ] - [ x_{i3}^{+} , x_{i0}^{-} ] = [ x_{i1}^{+} , x_{i2}^{-} ] - [ x_{i2}^{+} , x_{i1}^{-} ] = [ x_{i0}^{+} , x_{i3}^{-} ] - [ x_{i1}^{+} , x_{i2}^{-} ]. \]
	Thus we get
	\[ [ x_{i3}^{+} , x_{i0}^{-} ] = [ x_{i2}^{+} , x_{i1}^{-} ] = [ x_{i1}^{+} , x_{i2}^{-} ] = [ x_{i0}^{+} , x_{i3}^{-} ]. \]
	
	Consider the equation \eqref{eq:SYrc2}. The case $0 \le r \le 1$, $s=0$ follows from \eqref{eq:mpy4} and Lemmas \ref{lm:h0eq}, \ref{lm:hx20}. We prove by induction that \eqref{eq:SYrc2} holds for $0 \le r \le 1$ and all $s \in \mathbb{N}_{0}$. Suppose that \eqref{eq:SYrc2} holds for $0 \le r \le 1$ and $0 \le s \in \mathbb{N}_{0}$. Then we apply $[ \widetilde{h}_{i1} , \cdot]$ to get
	\[ [ \widetilde{h}_{i1} , [h_{i,r+1} , x_{is}^{\pm}] ] = [ \widetilde{h}_{i1} , [h_{ir} , x_{i,s+1}^{\pm}] \pm \frac{c_{ii} \hbar}{2} \{h_{ir}, x_{is}^{\pm}\} ] \Leftrightarrow \]
	\[ [h_{i,r+1} , x_{i,s+1}^{\pm}] = [h_{ir} , x_{i,s+2}^{\pm}] \pm \frac{c_{ii} \hbar}{2} \{h_{ir}, x_{i,s+1}^{\pm}\}. \]
	Thus by induction \eqref{eq:SYrc2} is true for $0 \le r \le 1$ and all $s \in \mathbb{N}_{0}$.
\end{proof}

\begin{lemma}
	\label{lm:cartelrelii}
	Let $i, j \in I$. Equations \eqref{eq:Cer} and \eqref{eq:SYrc3} hold for $i=j$ and all $r,s \in \mathbb{N}_{0}$.
\end{lemma}
\begin{proof}
	The proof is completely the same as in \cite{L93} (see considerations after Lemma 2.2) where the proof is based on mathematical induction. That's why we only explain here how to deal with equations (2.25)-(2.29) in \cite{L93} where some extra arguments are needed to deal with the case $|i| = \bar{1}$. First note that the base of induction is obtained by Lemma \ref{lm:r233}. Further in this proof in each equation we use the same notations as in mentioned paper. From now on suppose that $|i| = \bar{1}$, and if $c_{ii} = 0$ then $i^{'} = i+1$; otherwise $i^{'} = i-1$.

	Consider the equation (2.25). We are able to define $\dbtilde{h}_{i^{'}i,p}$ by Lemma \ref{lm:thfdpr} and to use Lemma \ref{lm:ijtrancart} to get
	\[ 0 = [ h_{i^{'}p} , h_{i^{'}p} ] = \frac{(\alpha_{i^{'}} , \alpha_{i^{'}})}{(\alpha_{i} , \alpha_{i^{'}})} [ h_{ip} , h_{i^{'}p} ] = [ h_{ip} , \dbtilde{h}_{i^{'}i,p} ]. \]
	Further steps are the same as in the paper. In the equation (2.26) we use $[\widetilde{h}_{i^{'}1}, \cdot]$. Further steps are the same.
	
	In the equation between (2.26) and (2.27):
	\[ [ h_{i^{'},r-q} , h_{i^{'}, q+1} ] = \frac{(\alpha_{i^{'}} , \alpha_{i^{'}})}{(\alpha_{i} , \alpha_{i^{'}})} [ h_{i,r-q} , h_{i^{'}, q+1} ] = \frac{(\alpha_{i^{'}} , \alpha_{i^{'}})}{(\alpha_{i} , \alpha_{i^{'}})} [ h_{i,r-q} , \dbtilde{h}_{i^{'}i,q+1} ]. \]
	Further steps are the same.
	
	Consider the equation (2.27). Here we use $[\widetilde{h}_{i^{'}1},
	\cdot]$. Further steps are the same. Consider the equation (2.28). We are able to define $\dbtilde{h}_{i^{'}i,p}$ by Lemma \ref{lm:thfdpr} and to use Lemma \ref{lm:ijtrancart} to get
	\[ [ h_{i^{'}, p+1} , h_{i^{'}p} ]  = \frac{(\alpha_{i^{'}} , \alpha_{i^{'}})}{(\alpha_{i} , \alpha_{i^{'}})} [ h_{i, p+1} , h_{i^{'}p} ] = \frac{(\alpha_{i^{'}} , \alpha_{i^{'}})}{(\alpha_{i} , \alpha_{i^{'}})} [ h_{i, p+1} , \dbtilde{h}_{i^{'}i,p} ]. \]
	Further steps are the same.
	
	In the equation (2.29) we use the same arguments as in the (2.28) and Lemma \ref{lm:hxinej} in the following to get
	\[ [ h_{i^{'}, p+1} , h_{i^{'}p} ] = \frac{(\alpha_{i^{'}} , \alpha_{i^{'}})}{(\alpha_{i} , \alpha_{i^{'}})} [ h_{i^{'},p+1} , h_{ip} ]. \]
	Further steps are the same. It is easy to see that all considerations after the equation (2.29) are also true in our case.
\end{proof}

\begin{lemma}
	\label{lm:cartelrelinej}
	Suppose that $i,j \in I$ and $i \ne j$. The equation \eqref{eq:Cer} holds for any $r,s \in \mathbb{N}_{0}$.
\end{lemma}
\begin{proof}
	The proof is the same as in \cite[Lemma 3.19]{MS22}.
\end{proof}

\begin{lemma}
	\label{lm:xijxijii}
	The relation \eqref{eq:SYrc4} is satisfied for all $i,j \in I$ and $r,s \in \mathbb{N}_{0}$.
\end{lemma}
\begin{proof}
	The case $i \ne j$ ($i,j \in I$) is proved in Lemma \ref{lm:SYrc34}. Suppose that $i=j$. Note that $|i| = \bar{0}$. We prove by induction on $r$ and $s \in \mathbb{N}_{0}$. The initial case $(r,s)=(0,0)$ is our initial assumption \eqref{eq:mpy5}. Let $X^{\pm}(r,s)$ be the result of substracting the right-hand side of \eqref{eq:SYrc4} from the left-hand side. We are able to define $\dbtilde{h}_{i^{'}i,r}$ ($i \in I$, $|i^{'} - i|=1$, $r \in \mathbb{N}_{0}$) by Lemma \ref{lm:thfdpr}. Using the relation \eqref{eq:thxo} we have for an arbitrary $r \in \mathbb{N}_{0}$:
	\[ 0 = [ \dbtilde{h}_{i^{'}i,r}, X^{\pm}(0,0) ] = \pm c_{i^{'} i} (  X^{\pm}(r,0) +  X^{\pm}(0,r) ) = \]
	\[ \pm c_{i^{'} i} 2 X^{\pm}(r,0) \Rightarrow X^{\pm}(r,0) = 0. \]
	Now for an arbitrary $s \in \mathbb{N}_{0}$:
	\[ 0 = [ \dbtilde{h}_{i^{'}i,s}, X^{\pm}(r,0) ] = \pm c_{i^{'} i} ( X^{\pm}(r+s,0) + X^{\pm}(r,s) ) = \]
	\[ \pm c_{i^{'} i} X^{\pm}(r,s) \Rightarrow X^{\pm}(r,s) = 0. \]
	The result follows by induction hypothesis.
\end{proof}

\begin{lemma}
	\label{lm:hixi}
	Let $i, j \in I$. The equation \eqref{eq:SYrc2} holds for all $i, j \in I$ and $r,s \in \mathbb{N}_{0}$.
\end{lemma}
\begin{proof}
	The case $i \ne j$ is proved in Lemma \ref{lm:hxinej}. Suppose that $i=j$. We prove by induction on $r$ and $s \in \mathbb{N}_{0}$. Let $X^{\pm}(r,s)$ be the result of substracting the right-hand side of \eqref{eq:SYrc2} from the left-hand side. The case $r=0$ and arbitrary $s \in \mathbb{N}_{0}$ follows by Lemma \ref{lm:h0eq}. Suppose that $r \ge 1$ and $X^{\pm}(r,s) = 0$ for all $s \in \mathbb{N}_{0}$. The case $r=1$ follows from Lemma \ref{lm:r233}. Note that $|i| = \bar{0}$. We consider the case $\pm = +$ (the case $\pm=-$ is proved in the same way). Apply $[ x^{-}_{i1}, \cdot ]$ to \eqref{eq:SYrc4} (we can do it by Lemma \ref{lm:xijxijii}) to get by Lemma \ref{lm:cartelrelii}
	\[ [ x^{-}_{i1}, [x_{i,r+1}^{+}, x_{i,s}^{+}] - [x_{i,r}^{+}, x_{i,s+1}^{+}] ] = \frac{c_{ii} \hbar}{2} [ x^{-}_{i1}, \{ x_{i,r}^{+} , x_{i,s}^{+} \} ] \iff \]
	\[ [ h_{i,s+1} , x_{i,r+1}^{+} ] - [ h_{i,r+2} , x_{i,s}^{+} ] - [ h_{i,s+2} , x_{ir}^{+} ] + [ h_{i,r+1} , x_{i,s+1}^{+} ] =  \]
	\[ \frac{c_{ii} \hbar}{2} (-1) ( \{ h_{i,r+1} , x_{is}^{+} \} + \{ h_{i,s+1} , x_{i,r}^{+} \} ) \iff \]
	\[ X^{+}(r+1,s) + X^{+}(s+1,r) = 0. \]
	Select $s=0$ and note that by Lemma \ref{lm:r233} $X^{+}(1,r) = 0$. Thus
	$ X^{+}(r+1,0) = 0$. Suppose that $ X^{\pm}(r+1,s) = 0$ for $s \ge 0$. Apply $[ \widetilde{h}_{i,1}, \cdot ]$ to $ X^{\pm}(r+1,s) = 0$. By Lemma \ref{lm:cartelrelii} we have
	\[ [ \widetilde{h}_{i1} , [h_{i,r+2} , x_{is}^{\pm}] ] = [ \widetilde{h}_{i1} , [h_{i,r+1} , x_{i,s+1}^{\pm}] \pm \frac{c_{ii} \hbar}{2} \{h_{i,r+1}, x_{is}^{\pm}\} ] \iff \]
	\[ [h_{i,r+2} , x_{i,s+1}^{\pm}] = [h_{i,r+1} , x_{i,s+2}^{\pm}] \pm \frac{c_{ii} \hbar}{2} \{h_{i,r+1}, x_{i,s+1}^{\pm}\}. \]
	Thus $ X^{\pm}(r+1,s+1) = 0$. The result follows by induction hypothesis.
\end{proof}

\begin{lemma}
	\label{lm:fassrс}
	Relation \eqref{eq:cssr} holds in the following cases:
	\begin{enumerate}
		\item $(r,s,t)=(0,0,z)$ ($z \in \mathbb{N}_{0}$);
		\item $(r,s,t)=(1,0,z)$ ($z \in \mathbb{N}_{0})$.
	\end{enumerate}
	Equation \eqref{eq:cssra} is satisfied for
	\begin{enumerate}
		\setcounter{enumi}{2}
		\item $(r_1, r_2, r_3, t) = (0,0,0,z)$ ($z \in \mathbb{N}_{0}$); \label{item:lemma8}
		\item $(r_1, r_2, r_3, t) = (1,0,0,z)$ ($z \in \mathbb{N}_{0}$).
	\end{enumerate}	
\end{lemma}
\begin{proof}
	The proof remains consistent with the one presented in \cite[Lemma 2.33]{GNW18}. It's important to highlight that the similarity in the proof arises from the construction of a system of homogeneous linear equations, which, in turn, results in only the trivial solution.
\end{proof}

\begin{lemma}
	\label{lm:bcssr}
	Relations \eqref{eq:bssr} and \eqref{eq:cssr} are satisfied for all $i,j \in I$ and $r,s,t \in \mathbb{N}_{0}$.
\end{lemma}
\begin{proof}
	The proof is the same as in \cite[Lemma 3.21]{MS22}.
\end{proof}

\begin{lemma}
	\label{lm:SYrc5}
	The relation \eqref{eq:SYrc5} is satisfied for all $i,j \in I$ and $r,s \in \mathbb{N}_{0}$.
\end{lemma}
\begin{proof}
	The proof is the same as in \cite[Lemma 3.22]{MS22}.
\end{proof}

\begin{lemma}
	\label{lm:cssra}
	The relation \eqref{eq:cssra} is satisfied for all $r_1, r_2, r_3, t \in \mathbb{N}_{0}$.
\end{lemma}
\begin{proof}
	Let $X^{\pm}(r_1,r_2,r_3,t)$ be the left-hand side of \eqref{eq:cssra}. We prove by induction on $r_1$, $r_2$, $r_3$, and $t \in \mathbb{N}_{0}$. The initial case $(r_1,r_2,r_3,t)=(0,0,0,0)$ is our initial assumption \eqref{eq:mpy8a}. We have proved in Lemma \ref{lm:fassrс} that $X^{\pm}(0,0,0,t)=0$ for all $t \in \mathbb{N}_{0}$. Note that for any $r,t \in \mathbb{N}_{0}$
	\[ 0 = [ \dbtilde{h}_{m+n-1,m+n,r} , X^{\pm}(0,0,0,t) ] = \sum_{z=t}^{t+r} a_{z} X^{\pm}(0,0,0,z) + \]
	\[ \pm (\alpha_{m+n-1}, \alpha_{m+n}) X^{\pm}(r,0,0,t) = \pm (\alpha_{m+n-1}, \alpha_{m+n}) X^{\pm}(r,0,0,t) \Rightarrow X^{\pm}(r,0,0,t) = 0, \]
	where $a_z \in \Bbbk$ ($r \le z \le t+r$).
	
	It follows that for any $r,s,t \in \mathbb{N}_{0}$
	\[ 0 = [\dbtilde{h}_{m+n-1,m+n,s} , X^{\pm}(r,0,0,t)] = \sum_{z=t}^{t+s} a_{z} X^{\pm} ( r,0,0,z ) + \]
	\[ \pm (\alpha_{m+n-1}, \alpha_{m+n}) ( X^{\pm} ( r+s,0,0,t ) + X^{\pm} ( r,s,0,t ) ) = \pm (\alpha_{m+n-1}, \alpha_{m+n}) X^{\pm} ( r,s,0,t ) \Rightarrow \]	
	\[ X^{\pm} ( r,s,0,t ) = 0, \]
	where $a_z \in \Bbbk$ ($r \le z \le t+s$).

	Suppose that for any $r, s, t \in \mathbb{N}_{0}$ and for $0 \le l \le u$ ($s \in \mathbb{N}_{0}$), $X^{\pm}(r,s,u,t)=0$. Then
	\[ 0 = [\widetilde{h}_{m+n,1} , X^{\pm}(r,s,u,t)] = (\alpha_{m+n}, \alpha_{m+n-1}) X^{\pm} ( r,s,u,t+1 ) + \]
	\[ (\alpha_{m+n}, \alpha_{m+n}) ( X^{\pm}(r+1,s,u,t) + X^{\pm} (r,s+1,u,t) + X^{\pm} (r,s,u+1,t) ) =  \]
	\[ (\alpha_{m+n}, \alpha_{m+n}) X^{\pm} (r,s,u+1,t) \Rightarrow X^{\pm} (r,s,u+1,t) = 0. \]
	Thus the result follows by induction hypothesis.
\end{proof}

\begin{lemma}
	\label{lm:leqmp}
	The relation \eqref{eq:qssr} is satisfied for all $j \in I$ and $r,s \in \mathbb{N}_{0}$.
\end{lemma}
\begin{proof}
	Let $X^{\pm}(r,0,0,s)$ be the left-hand side of \eqref{eq:qssr}. We prove by induction on $r$ and $s \in \mathbb{N}_{0}$. The initial case $(r,0,0,s)=(0,0,0,0)$ is our initial assumption \eqref{eq:mpy9}. Note that if we apply $[\widetilde{h}_{m1} , \cdot]$, $[\widetilde{h}_{n1} , \cdot]$ and $[\widetilde{h}_{k1} , \cdot]$ ($m,n,k \in I$) to $X^{\pm}(r,0,0,s)$ we get from \eqref{eq:h1eq}
	\[ 0 = (\alpha_{m}, \alpha_{j-1}) X^{\pm}(r+1,0,0,s) + (\alpha_{m}, \alpha_{j}) ( X^{\pm}(r,1,0,s) + X^{\pm}(r,0,1,s) ) + (\alpha_{m}, \alpha_{j+1}) X^{\pm}(r,0,0,s+1), \]
	\[ 0 = (\alpha_{n}, \alpha_{j-1}) X^{\pm}(r+1,0,0,s) + (\alpha_{n}, \alpha_{j}) ( X^{\pm}(r,1,0,s) + X^{\pm}(r,0,1,s) ) + (\alpha_{n}, \alpha_{j+1}) X^{\pm}(r,0,0,s+1), \]
	\[ 0 = (\alpha_{k}, \alpha_{j-1}) X^{\pm}(r+1,0,0,s) + (\alpha_{k}, \alpha_{j}) ( X^{\pm}(r,1,0,s) + X^{\pm}(r,0,1,s) ) + (\alpha_{k}, \alpha_{j+1}) X^{\pm}(r,0,0,s+1). \]
	Consider the matrix
	\[ A = \begin{pmatrix}(\alpha_{m},\alpha_{j-1}) & (\alpha_{m},\alpha_{j}) & (\alpha_{m},\alpha_{j+1})\\ (\alpha_{n},\alpha_{j-1}) & (\alpha_{n},\alpha_{j}) & (\alpha_{n},\alpha_{j+1}) \\ (\alpha_{k},\alpha_{j-1}) & (\alpha_{k},\alpha_{j}) & (\alpha_{k},\alpha_{j+1}) \end{pmatrix}\]
	When the determinant of $A$ is nonzero, we have $X^{\pm}(r+1,0,0,s)=X^{\pm}(r,0,0,s+1)=0$.
	In order to determine when the determinant of $A$ is nonzero depending on the grading of roots it is sufficient to consider the following Dynkin (sub)diagrams: \hyperref[fig:dynkin-diagrams-type1]{$\romannumeral1$}, \hyperref[fig:dynkin-diagrams-type2a]{$\romannumeral2a$} and \hyperref[fig:dynkin-diagrams-type2b]{$\romannumeral2b$}.
	
	Select $m=j-1, n=j, k=j+1$ for \hyperref[fig:dynkin-diagrams-type2a]{$\romannumeral2a$} and \hyperref[fig:dynkin-diagrams-type2b]{$\romannumeral2b$}. Then $\det(A) = - ( ( \alpha_{j-1} , \alpha_{j-1}) + ( \alpha_{j+1} , \alpha_{j+1} ) )$. It is easy to see that the determinant is nonzero in that case.
	
	Note that case \hyperref[fig:dynkin-diagrams-type1]{$\romannumeral1$} can occur only when $|I| > 3$.
	
	\begin{enumerate}
		\item  Select for \hyperref[fig:dynkin-diagrams-9]{case 1}
		
		$|j-1|=|j+1|=\bar{0} \Rightarrow m=j-2, n=j, k=j+1$;
		
		$|j-1|=\bar{1}, |j+1|=\bar{0} \Rightarrow m=j-1, n=j, k=j+1$;
		
		$|j-1|=\bar{0}, |j+1|=\bar{1} \Rightarrow m=j-1, n=j, k=j+1$;
		
		$|j-1|=|j+1|=\bar{1} \Rightarrow m=j-2, n=j, k=j+1$.
		
		\item 	Select for \hyperref[fig:dynkin-diagrams-10]{case 2}
		
		$|j-1|=|j+1|=\bar{0} \Rightarrow m=j+2, n=j, k=j+1$;
		
		$|j-1|=\bar{1}, |j+1|=\bar{0} \Rightarrow m=j-1, n=j, k=j+1$;
		
		$|j-1|=\bar{0}, |j+1|=\bar{1} \Rightarrow m=j-1, n=j, k=j+1$;
		
		$|j-1|=|j+1|=\bar{1} \Rightarrow m=j+2, n=j, k=j+1$.
	\end{enumerate}
	
	\begin{center}
		\begin{minipage}{0.4\textwidth}
			\centering	
			\begin{tikzpicture}
				\tikzset{/Dynkin diagram, root radius=.17cm}
				\dynkin[text style/.style={scale=1.0},labels*={j-2,j-1,j,j+1},edge length=1.7cm]A{xxtx}
			\end{tikzpicture}
			\captionsetup{type=figure}
			\captionof{figure}{case 1}
			\label{fig:dynkin-diagrams-9}
		\end{minipage}%
		
		\vspace{2em}
		
		\begin{minipage}{0.4\textwidth}
		\centering	
		\begin{tikzpicture}
			\tikzset{/Dynkin diagram, root radius=.17cm}
			\dynkin[text style/.style={scale=1.0},labels*={j-1,j,j+1,j+2},edge length=1.7cm]A{xtxx}
		\end{tikzpicture}
		\captionsetup{type=figure}
		\captionof{figure}{case 2}
		\label{fig:dynkin-diagrams-10}
	\end{minipage}%
	\end{center}
	
	It is easy to see that the determinant is nonzero in all these cases. The result follows by induction hypothesis.

\end{proof}

\end{document}